\documentclass{amsart}
\usepackage{amsmath,amsfonts,mathrsfs,enumerate,graphicx,subfig,wasysym,mathtools}%,txfonts}%,amssymb}
\usepackage{color} %I added, we can dismiss later if not needed.
\newtheorem{theorem}{Theorem}

\newtheorem{proposition}{Proposition}[section]

\DeclareMathOperator{\dist}{dist}

\DeclareMathOperator{\GL}{GL}

\DeclareMathOperator{\diam}{diam}

\newcommand{\triple}[1]{{\left\vert\kern-0.25ex\left\vert\kern-0.25ex\left\vert #1 
    \right\vert\kern-0.25ex\right\vert\kern-0.25ex\right\vert}}

\newcommand{\threebar}[1]{{\left\vert\kern-0.25ex\left\vert\kern-0.25ex\left\vert #1 
    \right\vert\kern-0.25ex\right\vert\kern-0.25ex\right\vert}}

\newcommand{\R}{\mathbb{R}}
\newcommand{\A}{\mathsf{A}}

\newcommand{\Z}{\mathbb{Z}}

\title[Marginal instability of linear cocycles]{A note on the marginal instability rates of two-dimensional linear cocycles}
\author{Ian D. Morris and Jonah Varney}
\address{I. D. Morris: School of Mathematical Sciences, Queen Mary University of London, Mile End Road, London E1 4NS, United Kingdom}
\email{i.morris@qmul.ac.uk }
\address{J. Varney: Mathematics Department, University of Surrey, Guildford GU2 7XH, United Kingdom}
\email{jonahvarney@gmail.com}
\begin{document}

\begin{abstract}
A theorem of Guglielmi and Zennaro implies that if the uniform norm growth of a locally constant $\GL_2(\R)$-cocycle on the full shift  is not exponential then it must be either bounded or linear, with no other possibilities occurring. We give an alternative proof of this result and demonstrate that its conclusions do not hold for Lipschitz continuous cocycles over the full shift on two symbols.
\end{abstract}
\maketitle
%\setcounter{section}{-1}
%\section{Current state of writing up}

%{\blue{Actually looks quite tight, just needs a careful read-through}}

\section{Introduction and statement of results}

Define $\Sigma_N:=\{1,\ldots,N\}^{\Z}$ and equip this set with the infinite product topology, with respect to which it is a compact metrisable topological space. Define $T \colon \Sigma_N \to \Sigma_N$ to be the shift transformation $T[(x_n)_{n \in \Z}]:=(x_{n+1})_{n \in \Z}$, which is a homeomorphism. If a continuous function $A \colon \Sigma_N \to \GL_d(\R)$ is specified, one may be interested in the growth of the sequence $(a_n)$ defined by
\[a_n:= \sup_{x \in \Sigma_N} \left\|A(T^{n-1}x)\cdots A(T x)A(x)\right\|.\]
This sequence  is easily seen to be \emph{submultiplicative} in the sense that $a_{n+m}\leq a_na_m$ for all $n,m \geq 1$, which guarantees the existence of the limit
\[\varrho(A):=\lim_{n \to \infty} \sup_{x \in \Sigma_N} \left\|A(T^{n-1}x)\cdots A(T x)A(x)\right\|^{\frac{1}{n}}.\]
By replacing $A$ with $\varrho(A)^{-1} \cdot A$ we may without loss of generality assume that $\varrho(A)=1$, and we will make this assumption for the remainder of this note. In this note we will be interested in the behaviour of the sequence $(a_n)$ in the reduced case $\varrho(A)=1$.
 
 Let us say that $A \colon \Sigma_N \to \GL_d(\R)$  is \emph{locally constant} if for $x=(x_n)_{n \in \Z}$ the matrix $A(x)$ is determined by the symbol $x_0$ only. In this case, if $\mathsf{A}$ denotes the range of the function $A$, then one simply has
 \begin{equation}\label{eq:simple}\sup_{x \in \Sigma_N} \left\|A(T^{n-1}x)\cdots A(T x)A(x)\right\|=\sup_{A_1,\ldots,A_n \in \mathsf{A}} \left\|A_n\cdots A_1\right\|.\end{equation} 
The case in which $A$ is locally constant has been studied extensively due to its relevance to marginally unstable discrete-time linear switching systems in control theory, and investigations of sequences $(a_n)$ of the above form may be found in numerous works such as \cite{ChMaSi12,Ju09,JuPrBl08,Mo22,PrJu15,Su08,Va22,VaMo22}. The same problem has also been studied in \cite{BeCoHa14,BeCoHa16} based on quite different motivations relating to the notion of $k$-regular sequences in symbolic dynamics. In the works just cited the simpler formulation \eqref{eq:simple} corresponding to the locally constant case is the only case studied, but the more general case in which $A$ is not assumed locally constant has been touched upon in the ergodic optimisation literature, notably \cite{BoGa19} in which criteria for $(a_n)$ to be a bounded sequence are investigated.

An early result describing some possible behaviours of such sequences $(a_n)$ is the following, which is essentially due to N. Guglielmi and M. Zennaro:
\begin{theorem}\label{th:guze}
Let $A\colon \Sigma_N \to \GL_2(\R)$ be locally constant and define $\A:=\{A(x) \colon x \in \Sigma_N\}$. Suppose that
\[\lim_{n \to \infty} \sup_{x \in \Sigma_N}  \left\|A(T^{n-1}x)\cdots A(x)\right\|^{\frac{1}{n}}=1.\] 
Then one of the following holds: either
\begin{equation}\label{eq:subadd-limit}\lim_{n \to \infty}\frac{1}{n} \sup_{x \in \Sigma_N} \left\|A(T^{n-1}x)\cdots A(x)\right\|>0,\end{equation}
or we instead have
\[\sup_{n \geq 1}  \sup_{x \in \Sigma_N} \left\|A(T^{n-1}x)\cdots A(x)\right\|<\infty.\]
Moreover, the first case occurs if and only if the semigroup generated by $\A$ contains a nontrivial Jordan matrix with unit determinant, if and only if both of the following two conditions are met: $\A$ is simultaneously triangularisable, and the set of matrices in $\A$ with determinant $\pm 1$ is nonempty and is not simultaneously diagonalisable.
\end{theorem}
We remark that the situation described in Theorem \ref{th:guze} is quite delicate: if the dimension of the linear maps is raised from $2$ to $3$, or if a shift over a compact infinite alphabet is allowed in place of the finite alphabet $\{1,\ldots,N\}$, then the conclusion no longer holds and the above sequences may grow at a rate strictly intermediate between linear growth and boundedness  (see \cite{GuZe01,Mo22,PrJu15}). In this article we give an alternative proof of the above result which is due to the second named author and which was previously presented in the thesis \cite{Va22}.  We remark that the actual existence of the limit \eqref{eq:subadd-limit} is a new contribution originating in this article: in \cite{GuZe01,Va22} it was shown that the limit inferior and limit superior of this sequence are finite and nonzero, but it was not shown that they are equal to one another.

The second contribution of this article is to show that if the condition of being locally constant is relaxed then the dichotomy asserted in Theorem \ref{th:guze} ceases to hold. We prove:
\begin{theorem}\label{th:main-2}
Let $T \colon \Sigma_2 \to \Sigma_2$ be the full shift on two symbols and let $d$ be any metric which generates the infinite product topology on $\Sigma_2$. Then there exist Lipschitz continuous functions $f,g \colon \Sigma_2 \to (0,1]$ and $\phi \colon \Sigma_2 \to \R$ such that the  function $A \colon \Sigma_2 \to \GL_2(\R)$ defined by
\[A(x):=\begin{pmatrix}
f(x)&\phi(x) \\
0&g(x)
\end{pmatrix}\]
satisfies
\[\lim_{n \to \infty}  \sup_{x \in \Sigma_2} \left\|A(T^{n-1}x)\cdots A(x)\right\|^{\frac{1}{n}} =1,\]
\[\lim_{n \to \infty} \frac{1}{n} \sup_{x \in \Sigma_2} \left\|A(T^{n-1}x)\cdots A(x)\right\| =0\]
and
\[\sup_{n \geq 1} \sup_{x \in \Sigma_2} \left\|A(T^{n-1}x)\cdots A(x)\right\| =\infty.\]
\end{theorem}
We emphasise that the metric $d$ is not assumed to have any properties other than generating the usual topology on $\Sigma_2$. When working with shift spaces it is usual to consider metrics on $\Sigma_2$ such that 
\[\max_{y \in \Sigma_2} \diam \left\{(x_n)_{n \in \Z} \in \Sigma_2 \colon x_i=y_i \text{ for all }i\text{ such that }|i|\leq n\right\} =O(\theta^n)\]
for some $\theta \in (0,1)$, but in Theorem \ref{th:main-2} this sequence may be allowed to tend to zero arbitrarily slowly or quickly. The functions $f, g, \phi$ may therefore be freely taken to be ``super-continuous'' in the sense of \cite{BoZh16,QuSi12}.

The proof of Theorem \ref{th:guze} is direct, and proceeds by considering the semigroup generated by the set $\{A(x)\colon x\in \Sigma_N\}$. The proof of Theorem \ref{th:main-2} is more technically subtle and makes use of ergodic optimisation. The two proofs are presented in sections \ref{se:two} and \ref{se:three} below.

\section{Proof of Theorem \ref{th:guze}}\label{se:two}

In view of the identity \eqref{eq:simple} it is sufficent to prove the following: if $\A$ is a finite set of real $2\times 2$ matrices which satisfies
\[\lim_{n \to \infty}  \max_{A_1,\ldots,A_n \in \A} \left\|A_n\cdots A_1\right\|^{\frac{1}{n}} =1,\]
then the limit
\[\lim_{n \to \infty}  \frac{1}{n}\max_{A_1,\ldots,A_n \in \A} \left\|A_n\cdots A_1\right\|\]
exists; if $\A$ is simultaneously upper triangularisable and the set $\{A\in \A \colon |\det A|=1\}$ is not simultaneously diagonalisable, then the above limit is nonzero; and if the two conditions just mentioned do not both hold, then the semigroup generated by $\A$ is bounded. 

We will begin the proof by establishing the boundedness of the semigroup generated by $\A$ in a certain special case. The following result is closely related to \cite[Lemma 5.1]{GuZe01} but its proof is entirely different:
\begin{proposition}\label{pr:ymm}
Let $\mathsf{A}$ be any finite set of $2\times 2$ real upper triangular matrices all of which have spectral radius at most $1$ and have determinant strictly less than $1$ in absolute value. Then the semigroup generated by $\mathsf{A}$ is bounded.
\end{proposition}
\begin{proof}
Clearly we may freely assume that $\mathsf{A}$ is nonempty. Since every element of $\mathsf{A}$ has spectral radius at most $1$, its diagonal entries are at most $1$ in absolute value. By the finiteness of $\mathsf{A}$ there exist $\beta \in [0,1)$ and $M\geq 0$ with the following property: for every $A \in \mathsf{A}$, one of the every diagonal entries of $A$ is at most $1$ in absolute value, the other diagonal entry is at most $\beta$ in absolute value, and the upper-right entry is at most $M$ in absolute value. (In particular we may take $\beta:=\max_{A \in \A} |\det A|^{1/2}$ and $M:=\max_{A\in \A}\|A\|$.) Let $\threebar{\cdot}_1$ denote the norm on $2 \times 2$ real matrices given by the sum of the absolute values of the matrix entries. If $A_1,\ldots,A_n$ are arbitrary real $2\times 2$ matrices, $A_1',\ldots,A_n'$ are non-negative $2\times 2$ matrices, and for every $i=1,\ldots,n$ every entry of $A_i'$ is  greater than or equal to the absolute value of the corresponding entry of $A_i$, then it is easily seen that
\[\threebar{A_n\cdots A_1}_1 \leq \threebar{A_n'\cdots A_1'}_1.\]
(This fact is particularly easily demonstrated in the context of upper-triangular matrices, which is the only case which we shall need.) Consequently, in order to demonstrate that
\[\sup_{n \geq 1} \max_{A_1,\ldots,A_n \in \mathsf{A}} \threebar{A_n\cdots A_1}_1<\infty\]
it is sufficient to demonstrate that
\begin{equation}\label{eq:bee-bound}\sup_{n \geq 1} \max_{i_1,\ldots,i_n \in \{1,2\}} \threebar{B_{i_n}\cdots B_{i_1}}_1<\infty\end{equation}
where
\[B_1:=\begin{pmatrix} 1&M\\0&\beta\end{pmatrix},\qquad B_2:=\begin{pmatrix} \beta &M\\0&1\end{pmatrix},\]
since every $A \in \mathsf{A}$ either has the absolute values of all of its entries less than or equal to the corresponding entry of $B_1$, or has the same property with respect to the matrix $B_2$. 

We therefore demonstrate \eqref{eq:bee-bound}. Fix an arbitrary $n \geq 1$ and consider a product $B_{i_n}\cdots B_{i_1}$ of the matrices $B_1$, $B_2$ which maximises the value of $\threebar{B_{i_n}\cdots B_{i_1}}_1$ among all products of $B_1$ and $B_2$ of length $n$. We claim that necessarily $B_{i_n}\cdots B_{i_1}=B_1^m B_2^{n-m}$ for some integer $m \in \{0,1,\ldots,n\}$. Suppose for a contradiction that this is not the case: then $B_{i_n}\cdots B_{i_1}$ may be factorised as
\[B_{i_n}\cdots B_{i_1}=B_{i_n}\cdots B_{i_{k+1}}B_2B_1B_{i_{k-1}}\cdots B_{i_1}= X_1 B_2B_1X_2,\]
say, where $X_1,X_2$ are invertible non-negative upper triangular matrices (one or both of which might be the identity matrix). Now, the matrix
\[Z:=B_1B_2-B_2B_1=\begin{pmatrix} 0&2(1-\beta)M\\0&0\end{pmatrix}\]
is non-negative and nonzero, so $X_1ZX_2$ is also non-negative and nonzero. It follows that
\[\threebar{X_1B_1B_2X_2}_1 =\threebar{X_1B_2B_1X_2 + X_1 ZX_2}_1>\threebar{X_1B_2B_1X_2}_1\]
using the non-negativity of all of these matrices and the fact that $X_1ZX_2$ is nonzero. Since $X_1B_1B_2X_2$ also has the form $B_{j_n}\cdots B_{j_1}$ for some $j_1,\ldots,j_n \in \{1,2\}$ this contradicts the presumed maximality of $\threebar{B_{i_n}\cdots B_{i_1}}_1$. The claim is proved. Since for every $n \geq 1$
\[\max_{i_1,\ldots,i_n \in \{1,2\}} \threebar{B_{i_n}\cdots B_{i_1}}_1 = \max_{0 \leq m \leq n} \threebar{B_1^m B_2^{n-m}}_1\]
by the preceding claim, it follows that
\begin{align*}\sup_{n \geq 1} \max_{i_1,\ldots,i_n \in \{1,2\}} \threebar{B_{i_n}\cdots B_{i_1}}_1 &\leq \sup_{n,m \geq 0} \threebar{B_1^n B_2^m}_1\\
&\leq \left(\sup_{n \geq 0}\threebar{B_1^n}_1\right)\left(\sup_{m \geq 0}\threebar{B_2^m}_1\right).\end{align*}
Since $B_1$ and $B_2$ are diagonalisable with spectral radius $1$, the sets $\{\threebar{B_1^n}_1 \colon n\geq 0\}$ and $\{\threebar{B_2^m}_1\colon n\geq 0\}$ are bounded, and the result follows.
\end{proof}

We now prove Theorem \ref{th:guze} in the form described at the beginning of the section. Suppose first that $\A$ is not simultaneously upper triangularisable. This is equivalent to the statement that there does not exist a basis for $\R^2$ whose first element is an eigenvector for every $A \in \A$, and this in turn is equivalent to the statement that no one-dimensional vector subspace of $\R^2$ is preserved by every element of $\A$. It is by now well established (see e.g. \cite[Theorem 2.2]{Ju09}) that under this last condition there necessarily exists a norm $\threebar{\cdot}$ on $\R^2$ such that $\max_{A \in \A}\threebar{Av}=\threebar{v}$ for every $A \in \A$, in which case in particular $\max_{A \in \A} \threebar{A}\leq1$ in the associated operator norm. This clearly implies that $\threebar{A}\leq 1$ for every $A$ in the semigroup generated by $\A$ and the result follows in this case. The existence of the limit \eqref{eq:subadd-limit} in this case is trivial.

For the remainder of the proof we may assume that $\A$ is simultaneously upper triangularisable. By an orthogonal change of basis for $\R^2$ we may assume without loss of generality (and without affecting either the existence or the value of the desired limit \eqref{eq:subadd-limit}) that every $A \in \mathsf{A}$ is upper triangular. If any $B \in \mathsf{A}$ had spectral radius strictly greater than $1$ then we would have
\[1=\lim_{n \to \infty} \max_{A_1,\ldots,A_n \in \mathsf{A}} \|A_n\cdots A_1\|^{\frac{1}{n}} \geq \lim_{n \to \infty} \|B^n\|^{\frac{1}{n}}>1\]
which is a contradiction, so every element of $\mathsf{A}$ has spectral radius at most $1$ and therefore the absolute values of its diagonal entries are at most $1$. The latter property is clearly inherited by all products of elements of $\mathsf{A}$. 

We now wish to show that the limit
\begin{equation}\label{eq:that-limit}\lim_{n \to \infty} \frac{1}{n}\max_{A_1,\ldots,A_n\in \A} \|A_n\cdots A_1\| \geq 0\end{equation}
exists. Define a seminorm $|\cdot|$ on the vector space of real $2\times 2$ matrices by defining $|A|$ to be the absolute value of the upper-right entry of $A$. An easy direct calculation demonstrates that if $A$ and $B$ are $2 \times 2$ upper-triangular matrices whose diagonal entries have absolute value at most $1$, then $|AB| \leq |A|+|B|$. It follows directly that the sequence
\[n\mapsto \max_{A_1,\ldots,A_n \in \mathsf{A}} |A_n\cdots A_1|\]
is subadditive and therefore the limit
\[\lim_{n \to \infty} \frac{1}{n}\max_{A_1,\ldots,A_n\in \A} |A_n\cdots A_1| \geq 0\]
exists. On the other hand, for every product 
\[A_n\cdots A_1=\begin{pmatrix} a&b\\ 0&c\end{pmatrix}\]
of elements of $\mathsf{A}$ we clearly have
\[ |b| \leq \left\|\begin{pmatrix}a&b\\0&c\end{pmatrix}\right\|\leq \left\|\begin{pmatrix} a&0\\0&c\end{pmatrix}\right\| + \left\|\begin{pmatrix} 0&b\\0&0\end{pmatrix}\right\|=\max\{|a|,|c|\}+|b|\leq 1+|b|\]
and therefore
\[\max_{A_1,\ldots,A_n\in \A} |A_n\cdots A_1| \leq \max_{A_1,\ldots,A_n\in\A} \|A_n\cdots A_1\|\leq 1+\max_{A_1,\ldots,A_n\in \A} |A_n\cdots A_1|\]
for every $n \geq 1$. We conclude that
\[\lim_{n \to \infty} \frac{1}{n}\max_{A_1,\ldots,A_n\in\A} \|A_n\cdots A_1\|=\lim_{n \to \infty} \frac{1}{n}\max_{A_1,\ldots,A_n\in\A} |A_n\cdots A_1|\]
and in particular the former limit exists as required.

We now demonstrate that either the limit \eqref{eq:that-limit} is positive, or the semigroup generated by $\mathsf{A}$ is bounded. Define
\[\mathsf{A}_0:=\{A \in \mathsf{A} \colon |\det A|<1\},\]
\[\mathsf{A}_1:=\{A \in \mathsf{A}\colon |\det A|=1\}.\]
We observe that every matrix in $\mathsf{A}_1$ is of one of the following three types: either it is equal to plus or minus the identity matrix; or it is a nontrivial Jordan matrix with determinant $1$; or it is an upper triangular matrix with determinant $-1$. If $\mathsf{A}_1$ contains a matrix $B$ of the second type then clearly
\[\max_{A_1,\ldots,A_n \in \mathsf{A}} \|A_n\cdots A_1\| \geq \|B^n\| \geq n|B|>0\]
for every $n \geq 1$. If it contains two matrices of the third type which are not scalar multiples of one another then it is easy to check that their product is a nontrivial Jordan matrix $B$, and therefore for all $n \geq1$
\[\max_{A_1,\ldots,A_{2n} \in \mathsf{A}} \|A_{2n}\cdots A_1\| \geq \|B^n\| \geq n|B|>0.\]
In either of these two cases it is obvious that $\mathsf{A}_1$ is not simultaneously diagonalisable and that the limit \eqref{eq:that-limit} is positive. If neither of these cases holds then for some upper triangular matrix $X \in \GL_2(\R)$ which has determinant $-1$ and diagonal entries $\pm 1$, we have $\mathsf{A}_1 \subseteq \{I, -I, X, -X\}$. It is clear that in this case $\mathsf{A}_1$ is simultaneously diagonalisable. Let $\mathsf{X}:=\{I, -I, X, -X\}$. By the Cayley-Hamilton Theorem $X^2=I$ and it follows easily that $\{I, -I, X, -X\}$ is a group. Consequently every element of the semigroup generated by $\mathsf{A}$ either is an element of $\mathsf{X}$ or is contained in the semigroup generated by the set
\[\hat{\mathsf{A}}:=\mathsf{A}_0 \cup \left\{BA \colon B \in \mathsf{X}\text{ and }A \in \mathsf{A}_0\right\}\cup \left\{AB \colon B \in \mathsf{X}\text{ and }A \in \mathsf{A}_0\right\}.\]
But $\hat{\mathsf{A}}$ satisfies the hypotheses of Proposition \ref{pr:ymm}, so the semigroup which it generates is bounded. Thus the semigroup generated by $\mathsf{A}$ is bounded. We observe that unboundedness held precisely in those cases in which the semigroup generated by $\A$ contained a nontrivial Jordan matrix with unit determinant. The proof is complete.

\section{Proof of Theorem \ref{th:main-2}}\label{se:three}

\subsection{Ergodic-theory preliminaries and a technical result}

We will deduce Theorem \ref{th:main-2} from a more general ergodic-theoretic statement. We begin with some necessary definitions. If $X$ is a compact metrisable topological space then we let $\mathcal{M}(X)$ denote the set of all Borel probability measures on $X$. By the Riesz Representation Theorem we may identify $\mathcal{M}(X)$ with the set of all non-negative elements of the unit sphere of $C(X)^*$, and we equip $\mathcal{M}(X)$ with the topology which it inherits as a subspace of $C(X)^*$ in its weak-* topology via this identification. This topology makes $\mathcal{M}(X)$ compact and metrisable, and in this topology the function $\mu \mapsto \int \phi\,d\mu$ is continuous for every $\phi \in C(X)$. 

If $T\colon X\to X$ is a continuous transformation of a compact metric space then we let $\mathcal{M}_T(X)\subseteq \mathcal{M}(X)$ denote the set of all $T$-invariant Borel probability measures on $X$ and $\mathcal{E}_T(X)\subseteq\mathcal{M}_T(X)$ the set of all such measures with respect to which $T$ is ergodic. Both sets are nonempty and the former is closed in the weak-* topology on $\mathcal{M}(X)$, hence is also a compact metrisable topological space. If $\phi \colon X\to\R$ is continuous we define $\beta(\phi):=\sup_{\mu \in \mathcal{M}_T(X)} \int\phi\,d\mu$. We also define $\mathcal{M}_{\max}(\phi)$ to be the set of all $\mu \in \mathcal{M}_T(X)$ which attain this supremum, i.e. which satisfy $\int \phi\,d\mu=\beta(\phi)$. The set $\mathcal{M}_{\max}(\phi)$ is nonempty by elementary considerations of compactness and continuity, and moreover has nonempty intersection with $\mathcal{E}_T(X)$ (see for example \cite[Proposition 2.4]{Je06}).

The following result will be applied to prove Theorem \ref{th:main-2}:
\begin{theorem}\label{th:tech}
Let $T \colon X \to X$ be a homeomorphism of a compact metric space. Let $f,g  \colon \Sigma \to (0,1]$ be continuous, let $\phi \colon \Sigma \to \R$ be continuous, and suppose that $ \beta(\log f)=\beta(\log g)=0$.  For every $x \in X$ define
\[A(x):=\begin{pmatrix} f(x) & \phi(x)\\ 0&g(x)\end{pmatrix} \in \GL_2(\R).\]
%Then 
%\[\lim_{n \to \infty} \sup_{x \in X} \left\|A(T^{n-1}x)\cdots A(x)\right\|^{\frac{1}{n}} = 1.\]
If $\mathcal{M}_{\max}(\log f)\cap\mathcal{M}_{\max}(\log g)\neq \emptyset$ then
\[\lim_{n \to \infty} \frac{1}{n} \sup_{x \in X}\left\| A(T^{n-1}x)\cdots A(x)\right\|=\sup_{\mu \in \mathcal{M}_{\max}(\log f)\cap\mathcal{M}_{\max}(\log g)} \left|\int \phi\,d\mu\right|\geq 0\]
and if $\mathcal{M}_{\max}(\log f)\cap\mathcal{M}_{\max}(\log g)= \emptyset$ then
\[\lim_{n \to \infty} \frac{1}{n} \sup_{x \in X}\left\| A(T^{n-1}x)\cdots A(x)\right\|=0.\]
\end{theorem}

\subsection{Overview of the proof of Theorem \ref{th:tech} and its relationship with earlier work}
The proof of Theorem \ref{th:tech} begins with the observation that for every $n \geq 1$ and $x \in X$, the product $A(T^{n-1}x)\cdots A(x)$ has the form
\[\begin{pmatrix} \prod_{j=0}^{n-1}f(T^jx)& \sum_{k=0}^{n-1} \left(\prod_{j=k+1}^{n-1} f(T^jx) \right)\phi(T^kx)\left(\prod_{j=0}^{k-1} g(T^jx)\right)\\ 0&\prod_{j=0}^{n-1}g(T^jx) \end{pmatrix}.\]
Since the diagonal entries necessarily belong to $(0,1]$, if the norm of this product is to grow to infinity then it must do so at the same rate as the off-diagonal term
\[\Phi_n(x):=\sum_{k=0}^{n-1} \left(\prod_{j=k+1}^{n-1} f(T^jx) \right)\phi(T^kx)\left(\prod_{j=0}^{k-1} g(T^jx)\right)\]
and therefore the problem reduces to showing that
\[\lim_{n \to \infty} \frac{1}{n}\sup_{x \in X}|\Phi_n(x)| =\sup_{\mu \in \mathcal{M}_{\max}(\log f)\cap\mathcal{M}_{\max}(\log g)} \left|\int \phi\,d\mu\right|\]
if $\mathcal{M}_{\max}(\log f)\cap\mathcal{M}_{\max}(\log f)$ is nonempty, and that the same limit is equal to zero otherwise. 

In the second named author's PhD thesis this problem was approached as follows. In order to obtain examples sufficient to prove Theorem \ref{th:main-2} it is enough to consider the case in which $f$ is the constant function, in which case we need only study the somewhat simpler expression
\begin{equation}\label{eq:v}\Psi_n(x):=\sum_{k=0}^{n-1} \phi(T^kx)\left(\prod_{j=0}^{k-1} g(T^jx)\right).\end{equation}
This expression may be seen as a weighted Birkhoff average reminiscent of those appearing in the Wiener-Wintner-type ergodic theorems found in such works as \cite{As03,SaWa07,Wa96,WiWi41} and can be studied using a modification of the strategy used in \cite{SaWa07,Wa96}. Specifically, one may re-express \eqref{eq:v} in terms of an extended dynamical system $T_g \colon X \times [0,1] \to X\times [0,1]$ defined by $T_g(x,y):=(Tx, g(x)y)$ and continuous function $\psi(x,y):=\phi(x)y$, obtaining by a simple induction
\[\Psi_n(x)=\sum_{k=0}^{n-1} \phi(T^kx)\left(\prod_{j=0}^{k-1} g(T^jx)\right)=\sum_{k=0}^{n-1}\psi(T_g^k(x,1))\]
and therefore
\[|\Psi_n(x)|=\sup_{y \in [0,1]}\left|\sum_{k=0}^{n-1}\psi(T_g^k(x,y))\right|.\]
The identity
\begin{align}\label{eq:herm}\lim_{n \to \infty}\frac{1}{n}\sup_{x \in X}\left| \Psi_n(x)\right|&=\lim_{n \to \infty} \frac{1}{n}\sup_{(x,y) \in X \times[0,1]} \left|\sum_{k=0}^{n-1} \psi(T_g^k(x,y))\right|\\\nonumber
& = \sup_{\mu \in \mathcal{M}_{T_g}(X\times [0,1])} \left|\int \psi\,d\mu\right|\end{align}
may then by obtained using appropriate results from the ergodic optimisation literature (see e.g. \cite[Proposition 2.2]{Je06}). By describing carefully the invariant measures of $T_g$ and relating them to invariant measures of $T$ a formula similar to that in Theorem \ref{th:tech} may be deduced. This approach can be developed further to allow for the possibility that $g$ takes values in $[-1,1]$, although in this case the resulting description of the limit
\[\lim_{n \to \infty} \frac{1}{n} \sup_{x \in X} \left\|A(T^{n-1}x)\cdots A(x)\right\|\]
becomes an inequality and not an equality, due to the difficulties in treating additive cancellations in \eqref{eq:v} arising from changes in the sign of $g$. When $T$ and $X$ have additional regularity properties Theorem \ref{th:tech} may also be extended to allow the condition $\sup g\leq 1$ to be removed, since by the use of results such as \cite[Theorem 4.7]{Je06} the condition $\sup g\leq 1$ can be obtained automatically from the condition $\beta(\log g)=0$ at the cost of a change of co-ordinates in $\R^2$ which depends continuously on the base point $x\in X$: see \cite[\S5]{Va22}.

 %2 Ergodic sets Bull. AMS 58 116?36 %Unem?ethodepourminorerles exposants de Lyapunov et quelques examples montrant le caract`ere local d'un th?eor`eme d'Arnold et de Moser sur le tore de dimension 2

In the treatment of Theorem \ref{th:tech} in this work, we will take a different approach by exploiting the fact that the functions $\Phi_n$ defined above have the subadditivity property
\[|\Phi_{n+m}| \leq |\Phi_n \circ T^m|+|\Phi_m|\]
and applying techniques from subadditive ergodic optimisation (specifically, from the appendix to \cite{Mo13}) rather than the additive techniques of \cite{Je06}. This has the advantages that it allows $f$ to be nonconstant and results in a shorter proof, but has the disadvantage that $f$ and $g$ are constrained to take values in $(0,1]$ and not in $[-1,1]$ as is the case in \cite{Va22}.

Proofs of the following standard result may be found in e.g. \cite{De75,KaWe82}.
\begin{theorem}[Subadditive ergodic theorem]
Let $T$ be an ergodic measure-preserving transformation of a probability space $(X,\mathcal{F},\mu)$ and let $(\psi_n)_{n=1}^\infty$ be a sequence of integrable functions $X \to \R$ such that $\psi_{n+m} \leq \psi_n\circ T^m + \psi_m$ a.e. for every $n,m\geq 1$. Then
\[\lim_{n \to \infty} \frac{1}{n}\psi_n(x) = \lim_{n \to \infty} \frac{1}{n} \int \psi_n\,d\mu = \inf_{n\geq 1} \frac{1}{n}\int \psi_n\,d\mu\]
for $\mu$-a.e. $x\in X$.
\end{theorem}
We also require the following subadditive analogue of \eqref{eq:herm} which may be found in the appendix to \cite{Mo13}. For some closely-related earlier results see also \cite{Sc98,StSt00}.
\begin{theorem}\label{th:boris}
Let $T \colon X \to X$ be a continuous transformation of a compact metric space and let $(\psi_n)_{n=1}^\infty$ be a sequence of continuous functions $X \to \R$ such that $\psi_{n+m}(x)\leq \psi_m(T^nx)+\psi_n(x)$ for every $x \in X$ and $n,m\geq 1$. Then
\begin{align*}\inf_{n \geq 1} \sup_{\mu \in \mathcal{M}_T(X)} \frac{1}{n}\int \psi_n\,d\mu &=\sup_{\mu \in \mathcal{M}_T(X)}\inf_{n \geq 1}  \frac{1}{n}\int \psi_n\,d\mu\\
&=  \inf_{n\geq 1}\sup_{x \in X} \frac{1}{n}\psi_n(x) = \sup_{x \in X}\inf_{n \geq 1} \frac{1}{n}\psi_n(x).\end{align*}
In the first three expressions the infimum over $n \geq 1$ is equal to the limit as $n \to \infty$ of the same expression. In all cases, every supremum over $\mu \in \mathcal{M}_T(X)$ is equal to the corresponding supremum over $\mu \in \mathcal{E}_T(X)$ and is attained by an element of $\mathcal{E}_T(X)$.
\end{theorem}
We now proceed with the proof of Theorem \ref{th:tech} and thence that of Theorem \ref{th:main-2}.

\subsection{Proof of Theorem \ref{th:tech}}
Consider the sequence of continuous functions $\Phi_n \colon X \to \R$ defined by
\[\Phi_n(x):=\sum_{j=0}^{n-1} f(T^{n-1}x)\cdots f(T^{j+1}x) \phi(T^jx) g(T^{j-1}x)\cdots g(x). \]
We note that the relation
\[\Phi_{n+m}(x) = \Phi_m(T^nx) g(T^{n-1}x)\cdots g(x) + f(T^{n+m-1}x)\cdots f(T^nx)\Phi_n(x)\]
is satisfied for every $x\in X$ and $n,m \geq 1$, and consequently
\[|\Phi_{n+m}(x)| \leq |\Phi_m(T^nx)|+|\Phi_n(x)|\]
for all such $n$, $m$ and $x$. It follows that Theorem \ref{th:boris} is applicable to the sequence of functions $|\Phi_n|$, and therefore
\[\lim_{n \to \infty} \frac{1}{n} \sup_{x \in X} |\Phi_n(x)|= \sup_{\mu \in \mathcal{E}_T(X)} \inf_{n \geq 1} \frac{1}{n}\int |\Phi_n|\,d\mu.\]
Now, for every $x \in X$ and $n \geq 1$ we have
\[A(T^{n-1}x) \cdots A(x)=\begin{pmatrix} \prod_{j=0}^{n-1}f(T^jx)& \Phi_n(x)\\ 0&\prod_{j=0}^{n-1}g(T^jx) \end{pmatrix}\]
and so in particular 
%\begin{align*}\left|\left\|A(T^{n-1}x)\cdots A(x)\right\|-|\Phi_n(x)|\right|&=\left|\left\|\begin{pmatrix} \prod_{j=0}^{n-1}f(T^jx)& \Phi_n(x)\\ 0&\prod_{j=0}^{n-1}g(T^jx) \end{pmatrix}\right\| -\left\|\begin{pmatrix} 0& \Phi_n(x)\\ 0&0\end{pmatrix}\right\| \right|\\
%& \leq \left\|\begin{pmatrix} \prod_{j=0}^{n-1}f(T^jx)& 0\\ 0&\prod_{j=0}^{n-1}g(T^jx) \end{pmatrix}\right\| \\
%&=\max\left\{\left| \prod_{j=0}^{n-1}f(T^jx)\right|, \left| \prod_{j=0}^{n-1}g(T^jx)\right|\right\}\leq 1 \end{align*}
\begin{align*}\MoveEqLeft[5]{\left|\left\|A(T^{n-1}x)\cdots A(x)\right\|-|\Phi_n(x)|\right|}&\\
&=\left|\left\|\begin{pmatrix} \prod_{j=0}^{n-1}f(T^jx)& \Phi_n(x)\\ 0&\prod_{j=0}^{n-1}g(T^jx) \end{pmatrix}\right\| -\left\|\begin{pmatrix} 0& \Phi_n(x)\\ 0&0\end{pmatrix}\right\| \right|\\
& \leq \left\|\begin{pmatrix} \prod_{j=0}^{n-1}f(T^jx)& 0\\ 0&\prod_{j=0}^{n-1}g(T^jx) \end{pmatrix}\right\| \\
&=\max\left\{\left| \prod_{j=0}^{n-1}f(T^jx)\right|, \left| \prod_{j=0}^{n-1}g(T^jx)\right|\right\}\leq 1 \end{align*}
by the reverse triangle inequality.  Consequently
\[\lim_{n \to \infty} \frac{1}{n} \sup_{x \in X}\left\| A(T^{n-1}x)\cdots A(x)\right\|=\lim_{n \to \infty} \frac{1}{n} \sup_{x \in X} \left|\Phi_n(x)\right|= \sup_{\mu \in \mathcal{E}_T(X)} \inf_{n \geq 1} \frac{1}{n}\int |\Phi_n|\,d\mu\]
and to prove the theorem we will evaluate the latter. We consider in turn the cases where $\mu \in \mathcal{E}_T(X)$ fails to belong to $\mathcal{M}_{\max}(\log g)$, fails to belong to $\mathcal{M}_{\max}(\log f)$, or belongs to both sets.

If $\mu\in\mathcal{E}_T(X)$ does not belong to $\mathcal{M}_{\max}(\log g)$ then for $\mu$-a.e. $x \in X$
\[\lim_{k \to \infty} \frac{1}{k} \log \left(g(T^{k-1}x)\cdots g(x)\right)=\lim_{k \to \infty} \frac{1}{k}\sum_{j=0}^{k-1} \log g(T^jx) = \int \log g\,d\mu<\beta(\log g)=0\]
and hence for $\mu$-a.e. $x \in X$
\begin{align*}|\Phi_n(x)|&=\left| \sum_{k=0}^{n-1} f(T^{n-1}x)\cdots f(T^{k+1}x) \phi(T^kx) g(T^{k-1}x)\cdots g(x)\right|\\
&\leq \sum_{k=0}^{n-1} \left|\phi(T^kx)\right| \cdot g(T^{k-1}x)\cdots g(x)\\
& \leq \|\phi\|_\infty\sum_{k=0}^\infty g(T^{k-1}x) \cdots g(x)<\infty\end{align*}
for every $n \geq 1$. It follows that $ \frac{1}{n}|\Phi_n|\to0$ $\mu$-a.e. and hence by the subadditive ergodic theorem
\[\lim_{n \to \infty} \frac{1}{n} \int |\Phi_n|\,d\mu = \inf_{n \geq1} \frac{1}{n} \int |\Phi_n|\,d\mu=0.\]
Now suppose instead that $\mu\in\mathcal{E}_T(X)$ does not belong to $\mathcal{M}_{\max}(\log f)$. Since $\mu$ is $T$-invariant it is also $T^{-1}$-invariant, so by the Birkhoff ergodic theorem applied to $T^{-1}$ we likewise have
\[\lim_{\ell \to \infty} \frac{1}{\ell} \log f(x)f(T^{-1}x)\cdots f(T^{-(\ell-1)}x)=\int \log f\,d\mu<\beta(\log f)=0\]
and hence for $\mu$-a.e. $x \in X$
\begin{align*}\left|\Phi_n(T^{-(n-1)}x)\right|&=\left| \sum_{k=0}^{n-1} f(x)\cdots f(T^{k+2-n}x) \phi(T^{k+1-n}x) g(T^{k-n}x)\cdots g(T^{-(n-1)}x)\right|\\
&=\left| \sum_{\ell=0}^{n-1} f(x)\cdots f(T^{-(\ell-1)}x) \phi(T^{-\ell}x) g(T^{-\ell-1}x)\cdots g(T^{-(n-1)}x)\right|\\
&\leq  \sum_{\ell=1}^{n} f(x)\cdots f(T^{-(\ell-1)}x)\cdot\left|\phi(T^{-\ell}x)\right|\\
&\leq \|\phi\|_\infty \sum_{\ell=1}^\infty  f(x)\cdots f(T^{-(\ell-1)}x)<\infty\end{align*}
for every $n \geq 1$. It follows that for $\mu$-a.e. $x\in X$
\[\lim_{n \to \infty} \frac{1}{n}\left|\Phi_n\left(T^{-(n-1)}x\right)\right|=0\]
and hence in particular $\frac{1}{n}|\Phi_n \circ T^{-(n-1)}| \to 0$ in the sense of convergence in measure; but since $\mu$ is $T$-invariant, this implies that $\frac{1}{n} |\Phi_n| \to 0$ in the sense of convergence in measure. In combination with the subadditive ergodic theorem this fact yields
\[\lim_{n \to \infty} \frac{1}{n} \int |\Phi_n|\,d\mu = \inf_{n \geq1} \frac{1}{n} \int |\Phi_n|\,d\mu=0.\]
Combining the two facts just demonstrated, we have shown that if $\mu\in\mathcal{E}_T(X)$ does not belong to $\mathcal{M}_{\max}(\log f) \cap \mathcal{M}_{\max}(\log g)$ then
\[\inf_{n \geq 1} \frac{1}{n} \int |\Phi_n|\,d\mu=\lim_{n \to \infty} \frac{1}{n} \int |\Phi_n|\,d\mu = 0.\]
If $\mathcal{M}_{\max}(\log f) \cap \mathcal{M}_{\max}(\log g)$ is empty, this demonstrates that
\[ \sup_{\mu \in \mathcal{E}_T(X)}\inf_{n \geq 1}\frac{1}{n} \int |\Phi_n|\,d\mu =0\]
which completes the proof of the theorem in that case. Otherwise, suppose that $\mathcal{M}_{\max}(\log f) \cap \mathcal{M}_{\max}(\log g)$ is nonempty. Since by hypothesis $0= \beta( \log f)\leq \sup \log f \leq 0$ and $0=\beta(\log g) \leq  \sup \log g \leq 0$ it is easily seen that the set
\[Z:= \left\{x \in X \colon \log f(x)=\log g(x)=0\right\}=\left\{x \in X \colon f(x)=g(x)=1\right\}\]
satisfies $\mu(Z)=1$ for every $\mu \in \mathcal{M}_{\max}(\log f)\cap \mathcal{M}_{\max}(\log g)$ and in particular is nonempty. Moreover it is easily verified that
\[\mathcal{M}_{\max}(\log f)\cap \mathcal{M}_{\max}(\log g) = \left\{ \mu \in \mathcal{M}_T(X) \colon \mu(Z)=1\right\}=\mathcal{M}_T(Z)\]
and therefore
\[\mathcal{M}_{\max}(\log f)\cap \mathcal{M}_{\max}(\log g)\cap \mathcal{E}_T(X) = \mathcal{M}_T(Z)\cap\mathcal{E}_T(X)=\mathcal{E}_T(Z).\]
Now, if $\mu \in \mathcal{M}_{\max}(\log f) \cap \mathcal{M}_{\max}(\log g)\cap \mathcal{E}_T(X)= \mathcal{E}_T(Z)$ then
for $\mu$-a.e. $x \in X$ we simply have
\[\lim_{n \to \infty} \frac{1}{n} |\Phi_n(x)| = \lim_{n \to \infty} \frac{1}{n} \left|\sum_{k=0}^{n-1} \phi(T^kx)\right| = \left|\int \phi\,d\mu\right|\]
using the Birkhoff ergodic theorem and the fact that $f$ and $g$ are identically equal to $1$ on $Z$. Thus
\[\sup_{\mu \in \mathcal{M}_{\max}(\log f)\cap\mathcal{M}_{\max}(\log g)\cap  \mathcal{E}_T(X)}\left|\int \phi\,d\mu\right|\\
=\sup_{\mu \in \mathcal{E}_T(Z)}\left|\int \phi\,d\mu\right|.\]
 Since we have already established that
\[\sup_{\mu \in \mathcal{E}_T(X) \setminus (\mathcal{M}_{\max}(\log f)\cap\mathcal{M}_{\max}(\log g))} \inf_{n \geq 1} \frac{ 1}{n} \int |\Phi_n|\,d\mu=0\]
%for every $\mu \in \mathcal{E}_T(X) \setminus (\mathcal{M}_{\max}(\log f)\cap\mathcal{M}_{\max}(\log g))$,
 it follows that
\begin{align*}
\sup_{\mu \in \mathcal{E}_T(X)} \inf_{n \geq 1}\frac{1}{n} \int |\Phi_n|\,d\mu &=\sup_{\mu \in \mathcal{M}_{\max}(\log f)\cap\mathcal{M}_{\max}(\log g)\cap  \mathcal{E}_T(X)}\left|\int \phi\,d\mu\right|\\
&=\sup_{\mu \in \mathcal{E}_T(Z)}\left|\int \phi\,d\mu\right|\\
&=\sup_{\mu \in \mathcal{M}_T(Z)}\left|\int \phi\,d\mu\right|\\
&=\sup_{\mu \in \mathcal{M}_{\max}(\log f)\cap\mathcal{M}_{\max}(\log g)}\left|\int \phi\,d\mu\right|\end{align*}
as required. The proof of the theorem is complete.

 \subsection{Proof of Theorem \ref{th:main-2}}
Fix a metric $d$ on $\Sigma_2$ which generates the infinite product topology. Let $[1]\subset \Sigma_2$ denote the set of all $(x_n)_{n\in \Z} \in \Sigma_2$ such that $x_0=1$, which by the definition of the infinite product topology on $\Sigma_2$ is both closed and open. By an appropriate version of the Jewett-Krieger Theorem (see for example \cite[\S29]{DeGrSi76}), or by various direct constructions such as \cite{DeKe78,Gr73,GrSh76}, there exists a compact $T$-invariant set $Z \subset \Sigma_2$ with the following properties: there exists a unique $\nu \in \mathcal{M}_T$ such that $\nu(Z)=1$; $T$ is weak-mixing with respect to this unique measure $\nu$; the support of $\nu$ is precisely $Z$; and $Z$ is not a singleton set. Since $Z$ is not a singleton we have $0<\nu([1])<1$ and therefore $e^{2\pi i \nu([1])} \neq 1$, a property which will be significant later. Define $f(x)=g(x)=e^{-\dist(x,Z)}$ for all $x \in \Sigma_2$, where $\dist(x,Z):=\inf_{y \in Z}d(x,y)$. Clearly $f$ and $g$ are Lipschitz continuous and satisfy $\beta(\log f)=\beta(\log g)=0$ and $\mathcal{M}_{\max}(\log f)=\mathcal{M}_{\max}(\log g)=\{\nu\}$. Define also $\phi(x):=\chi_{[1]}(x)-\nu([1])$. By the definition of the infinite product topology, $\phi \colon \Sigma_2 \to \R$ is continuous; since $\Sigma_2$ is a compact metric space with respect to $d$, $\phi$ is uniformly continuous with respect to $d$; and since $\phi$ takes exactly two values, this implies that $\phi$ is Lipschitz continuous with respect to $d$ as required. Define $A \colon\Sigma_2\to \GL_2(\R)$ as in the statement of the theorem. Clearly  Theorem \ref{th:tech} is applicable. Since
\[\sup_{\mu \in \mathcal{M}_{\max}(\log f)\cap \mathcal{M}_{\max}(\log g)} \left|\int \phi\,d\mu\right|=\left|\int \phi\,d\nu\right|=\left|\nu([1])-\nu([1])\right|=0,\]
Theorem \ref{th:tech} yields 
\begin{equation}\label{eq:this}\lim_{n \to \infty} \frac{1}{n} \sup_{x \in \Sigma_2} \left\|A(T^{n-1}x)\cdots A(x)\right\| =0.\end{equation}
Suppose for a contradiction that 
\[\sup_{n \geq 1} \sup_{x \in \Sigma_2} \left\|A(T^{n-1}x)\cdots A(x)\right\|<\infty,\]
in which case since $f \equiv g \equiv 1$ on $Z$,
\begin{align*}\sup_{n \geq 1} \sup_{x \in Z} \left|\sum_{j=0}^{n-1} \phi(T^jx) \right|
&\leq \sup_{n \geq 1} \sup_{x \in Z}\left\|\begin{pmatrix} 1& \sum_{j=0}^{n-1} \phi(T^jx) \\0&1\end{pmatrix}\right\|\\
&= \sup_{n \geq 1} \sup_{x \in Z}\left\|\begin{pmatrix} 1& \phi(T^{n-1}x) \\0&1\end{pmatrix}\cdots \begin{pmatrix} 1& \phi(Tx) \\0&1\end{pmatrix} \begin{pmatrix} 1& \phi(x) \\0&1\end{pmatrix} \right\|\\
%&=\sup_{n \geq 1} \sup_{x \in Z} \left|\sum_{j=0}^{n-1} f(T^{n-1}x)\cdots f(T^{j+1}x)\phi(T^jx)g(T^{j-1}x)\cdots g(x) \right|\\
&= \sup_{n \geq 1}\sup_{x \in Z}  \left\|A(T^{n-1}x)\cdots A(x)\right\|\\
&\leq \sup_{n \geq 1} \sup_{x \in \Sigma_2} \left\|A(T^{n-1}x)\cdots A(x)\right\|<\infty.\end{align*}
We borrow an argument of Hal\'asz \cite{Ha76} to deduce that $T$ cannot be weak mixing with respect to $\nu$, giving us the required contradiction. 
In view of the above bound we may define a bounded Borel measurable function $\psi \colon Z \to \R$ by $\psi(x):=\limsup_{n \to\infty} \sum_{j=0}^{n-1}\phi(T^jx)$. Clearly $\psi(x) =\phi(x)+\psi(Tx)$ for every $x\in Z$, so
\[e^{2\pi i\psi(x)}=e^{2\pi i\phi(x)}e^{2\pi i\psi(Tx)}=e^{2\pi i \chi_{[1]}(x)}e^{-2\pi i\nu([1])}e^{2\pi i\psi(Tx)}=e^{-2\pi i\nu([1])}e^{2\pi i\psi(Tx)}\]
for every $x \in Z$, where we have used the fact that the function $\chi_{[1]}$ takes only integer values. In particular $e^{2\pi i \psi}\circ T=e^{2\pi i\nu([1])}e^{2\pi i \psi}$ $\nu$-a.e, so $e^{2\pi i \psi}$ is an eigenfunction of the composition operator $h \mapsto h \circ T$ on $L^2(\nu)$ with eigenvalue $e^{2\pi i \nu([1])} \neq 1$, which contradicts the fact that $T$ is weak-mixing with respect to $\nu$. We have obtained the desired contradiction and deduce that necessarily
\begin{equation}\label{eq:that}\sup_{n \geq 1} \sup_{x \in \Sigma_2} \left\|A(T^{n-1}x)\cdots A(x)\right\|=\infty.\end{equation}
The limit
\[\lim_{n \to \infty} \frac{1}{n}\log \sup_{x \in \Sigma_2} \left\|A(T^{n-1}x)\cdots A(x)\right\|\]
clearly exists by subadditivity. In view of \eqref{eq:this} this limit cannot be strictly greater than zero, and  in view of \eqref{eq:that} it cannot be strictly less than zero. It is therefore zero, which is to say that
\[\lim_{n \to \infty}  \sup_{x \in \Sigma_2} \left\|A(T^{n-1}x)\cdots A(x)\right\|^{\frac{1}{n}}=1\]
as required by the statement of the theorem. The proof of the theorem is complete.
\section{Acknowledgements}

Versions of Theorems \ref{th:guze}, \ref{th:main-2} and \ref{th:tech}, and of Proposition \ref{pr:ymm}, previously appeared in the second named author's PhD thesis \cite{Va22}. J. Varney was supported by EPRSC Doctoral Training Partnership grant EP/R513350/1. I.D. Morris was partially supported by Leverhulme Trust Research Project Grant RPG-2016-194.

\bibliographystyle{acm}
\bibliography{vsb}

\end{document}